\documentclass[12pt]{amsart}
\usepackage{graphicx}
\usepackage{amssymb}
\usepackage{verbatim}
\usepackage{array}
\usepackage{latexsym}
\usepackage{enumerate}
\usepackage{amsmath}
\usepackage{amsfonts}
\usepackage{amsthm}
\usepackage[english]{babel}
\usepackage{color}

\newcommand{\FF}{\mathbb F}

\newcommand{\KK}{\mathbb K}

\def\cC{\mathcal C}

\def\cH{\mathcal H}
\def\cK{\mathcal K}

\def\cO{\mathcal O}

\def\cW{\mathcal W}
\def\cX{\mathcal X}
\def\cY{\mathcal Y}

\def\mod{\mbox{\rm mod}}
\def\deg{\mbox{\rm deg}}

\def\div{\mbox{\rm div}}

\def\fqq{{\mathbb F}_{q^2}}

\def\fqs{{\mathbb F}_{q^2}}

\def\fns{{\mathbb F}_{n^2}}

\newcommand{\aut}{\mbox{\rm Aut}}

\newcommand{\ga}{\alpha}
\newcommand{\gb}{\beta}
\newcommand{\g}{\gamma}
\newcommand{\gd}{\delta}

\newcommand{\gl}{\lambda}

\newcommand{\gr}{\rho}

\newtheorem{theorem}{Theorem}[section]
\newtheorem{proposition}[theorem]{Proposition}

\newtheorem{remark}[theorem]{Remark}

{\theoremstyle{definition}
\newtheorem*{definition*}{Definition}

\newtheorem{case}[theorem]{Case}

\newtheorem*{proposition*}{Proposition}
\newtheorem*{corollary*}{Corollary}
\newtheorem*{lemma*}{Lemma}
\begin{document}

\title{Maximal curves from subcovers of the GK-curve}

\thanks{{\bf Keywords}: Maximal curves, AG-codes, GK-curve}

\thanks{{\bf Mathematics Subject Classification (2010)}: 11G20}

\thanks{Research supported by  the Italian
    Ministry MIUR, Strutture Geometriche, Combinatoria e loro Applicazioni, PRIN 2012 prot. 2012XZE22K, and by INdAM. The second author was partially supported by Siemens, project INST.MAT.-18.414.}

\author{Massimo Giulietti}
\address[M. Giulietti]{Dipartimento di Matematica e Informatica \\ Universit\`a degli Studi di Perugia \\ Via Vanvitelli, 1 \\ 06123 Perugia - Italy }
\email{massimo.giulietti@unipg.it}

%\author{Massimo Giulietti, Luciane Quoos, and Giovanni Zini}
%\address[M. Giulietti]
\author{Luciane Quoos}
\address[L. Quoos]{Instituto de Matem\'atica \\ Universidade Federal do Rio de Janeiro \\ Cidade Universit\'aria \\ Rio de Janeiro, RJ 21941-909\\ Brazil }
%%\email[L. Quoos]{luciane@im.ufrj.br}
\email{luciane@im.ufrj.br}

\author{Giovanni Zini}
\address[G. Zini]{Dipartimento di Matematica e Informatica ``Ulisse Dini'' \\ Universit\`a degli Studi di Firenze \\ Viale Morgagni, 67/a \\ 50134 Firenze - Italy}
\email{gzini@math.unifi.it}

\begin{abstract}
For every $q=n^3$
with $n$ a prime power greater than $2$, the GK-curve is an $\fqq$-maximal curve that is not
$\fqq$-covered by the Hermitian curve.
In this paper some Galois subcovers of the GK curve are investigated.  We describe explicit equations for some families of quotients of the GK-curve. New values in the spectrum of genera of $\fqs$-maximal curves are obtained. Finally, infinitely many further examples of maximal curves that cannot be Galois covered by the Hermitian curve are provided.
\end{abstract}

\maketitle

\section{Introduction}
Let $\fqq$ be a finite field with $q^2$ elements, where $q$ is a
power of a prime $p$. By a curve $\cX$ over $\fqq$ we mean a
projective, absolutely irreducible, non-singular algebraic curve defined over $\fqq$, and it is called $\fqq$-maximal if the number $\cX(\fqs)$ of its
$\fqq$-rational points attains the  Hasse-Weil upper bound
$$q^2+1+2gq,$$
where $g$ is the genus of the curve. Maximal curves have
interesting properties and have also been investigated for their
applications in Coding Theory: sometimes the best known linear codes over finite fields of square order are obtained as one-point AG-codes from maximal curves, see e.g. \cite{Fanali,NW2,MINT}.

One of the most important problems about maximal curves is determining the possible genera of maximal curves over $\fqs$. For a given $q$, the highest value of $g$ for which an $\fqs$-maximal curve of genus $g$ exists is $q(q-1)/2$, and equality holds if and only if 
the curve is isomorphic over $\fqs$ to the Hermitian curve. 
%The upper part in the spectrum of genera of $\fqs$-maximal curves contains only few values, and although the whole spectrum depends on
%the nature of $q$, the highest three genera in the spectrum have the same positions
%independently of $q$ \cite[Sect. 10.5]{hirschfeld-korchmaros-torres2008}.
By a result of Serre, cited by Lachaud in \cite{L}, any $\fqs$-rational curve which is $\fqs$-covered by an $\fqs$-maximal curve is also $\fqs$-maximal.
However, not every maximal curve can be covered by the Hermitian curve. In fact,
in 2009 Giulietti and Korchm\'aros constructed a maximal
curve over $\mathbb F_{n^6}$ which cannot be covered by the Hermitian curve whenever $n>2$; this curve is nowadays referred to as the GK-curve.

Serre's covering result has made it possible to obtain several genera of
$\fqs$-maximal curves by applying the Riemann-Hurwitz formula to quotient curves of a known $\fqs$-maximal curve, such as  the Hermitian curve, the Suzuki curve, the Ree curve, and the GK-curve, see \cite{AQ, FG, GSX, GTK}.
However, sometimes it can be hard to give {\em explicit equations} for a quotient curve. This problem is clearly relevant for applications to Coding Theory, but has been attacked only recently for the GK-curve. Equations of some quotients of the GK-curve with respect to $p$-groups of automorphisms  have been obtained in \cite{SP,TTT}.

In this paper we deal with Galois subcovers of the GK-curve with respect to coverings of
%MG
 degree not divisible by $p$. 
We provide explicit equations for several such subcovers, see Theorem \ref{equations} and Equation \eqref{XsuK}, and in some cases we give an explicit description of the Galois group of the covering; see Section \ref{quattro}. Our starting point for computing such equations is a new non-singular model of the GK-curve; see Section \ref{NewModel}. The genera of these subcovers are computed, see Theorem \ref{genera} and Formula \eqref{gXsuK}, and sometimes they are new values in the spectrum of genera of $\FF_{n^6}$-maximal curves; see Remark \ref{61}. Interestingly, it often happens that these curves are not Galois covered by the Hermitian curve; see Theorems \ref{NotGaloisCovered}, \ref{NotGaloisCoveredII}, \ref{NotGaloisCoveredIII}. In some cases we are able to show they are not covered by the Hermitian curve at all; see Table \ref{tabella}.

\section{A new model of the GK-curve}\label{NewModel}

Let $p$ be a prime, $n$ a power of $p$, $q=n^3$, $\mathbb{K}=\bar{\FF}_{q^2}$ the algebraic closure of $\fqs$. Let $\cC$ be the so-called GK curve, which is $\fqs$-maximal and is defined by the affine equations
$$ \cC: \left\{
\begin{array}{c}
Z^{n^2-n+1}=Y\frac{X^{n^2}-X}{X^{n}+X}\\
Y^{n+1}=X^{n}+X \quad\;\,
\end{array} \right. .$$
Let $\rho\in\fns$ with $\rho+\rho^n=1$. Consider the $\fns$-projectivity $\varphi$ associated to the matrix $A$, where
$$
A=\begin{pmatrix} 1 & 0 & 0 & 1-\gr \\ 0 & 1 & 0 & 0 \\ 0 & 0 & -1 & 0 \\ 1 & 0 & 0 &-\gr \end{pmatrix}%,
%\qquad  A^{-1}=\begin{pmatrix} \gr & 0 & 0 & \gr^n \\ 0 & 1 & 0 & 0 \\ 0 & 0 & -1 & 0 \\ 1 & 0 & 0 &-1 \end{pmatrix}
.
$$
Then $\cX=\varphi(\cC)$ has equations
$$ \cX: \left\{
\begin{array}{c}
Z^{n^2-n+1}=Y\frac{X^{n^2}-X}{X^{n+1}-1}\\
Y^{n+1}=X^{n+1}-1 \quad\;
\end{array} \right. .$$
%The $\fqs$-maximality of $\cX'$ can be prove independently, which is a consequence of the Natural Embedding Theorem %(\cite[Th. 3.6]{KT}).
%
%\begin{proposition}
%Let $a,c\in\fq$, $b\in\fqs$ with $ac-b^{q+1}\neq0$. Let
%$$g(X)=aX^{q+1}+(b^q+b)X^q+bX+c$$
%and $f(X)\in\fqs[X]$ a divisor of $g(X)$ with $\deg(f)\leq d$, where $d$ is a divisor of $q+1$. If
%$$ f(X)^{\frac{q+1}{d}-1}-\frac{g(X)}{f(X)} $$
%is the $d$-th power of a polynomial $h(X)\in\fqs[X]$, then the curve with equations
%$$ \begin{cases} Z^{\frac{q+1}{d}}=Y\,h(X) \\ Y^d=f(X) \end{cases} $$
%is $\fqs$-maximal
%\end{proposition}

%This criterion follows from the Natural Embedding Theorem (\cite[Th. 3.6]{KT}).

We will consider subgroups of the following tame $\fqs$-automorphism group $G$ of $\cX$ of size $(n+1)^2(n^2-n+1)$:
\begin{equation}\label{groupG}
 G=\left\{g_{a,b,\gl}:(X,Y,Z,T)\mapsto(aX,bY,\gl Z,T)\mid a^{n+1}=b^{n+1}=1,\gl^{n^2-n+1}=ab\right\}.
\end{equation}
%$G'$ is tame, as $|G'|=(n+1)^2(n^2-n+1)$.

By conjugation, an $\fqs$-automorphism group $G^A=A^{-1}G A$ of $\cC$
%MG
 is obtained:
$$ G^A=\left\{g_{a,b,\gl}^A
\mid a^{n+1}=b^{n+1}=1,\gl^{n^2-n+1}=ab \right\},\quad \textrm{where} $$
$$ g_{a,b,\gl}^A=
\begin{pmatrix} a\rho+\rho^n & 0 & 0 & a\rho-a\rho^2-\rho^{n+1} \\ 0 & b & 0 & 0 \\ 0 & 0 & \lambda & 0 \\ a-1 & 0 & 0 & a-a\rho+\rho \end{pmatrix}. $$

According to the notation of \cite{FG}, we compute the projection  $\overline{G^A}$ of $G^A$ over $\textrm{PGU}(3,n)$ and the intersection $G_\Lambda^A$ of $G^A$ with
\begin{equation}\label{Lambda}
\Lambda=\left\{\alpha_{\lambda}:(X,Y,Z,T)\mapsto(X,Y,\lambda Z,T)\mid \gl^{n^2-n+1}=1 \right\}:
\end{equation}
$$ G_\Lambda^A = \Lambda,\quad \overline{G^A}=\left\{ \bar{g}_{a,b} \mid a^{n+1}=b^{n+1}=1 \right\},\quad \textrm{where} $$
$$ \bar{g}_{a,b}=\begin{pmatrix} a\rho+\rho^n & 0 & a\rho-a\rho^2-\rho^{n+1} \\ 0 & b & 0 \\ a-1 & 0 & a-a\rho+\rho \end{pmatrix}. $$

Note that $\overline{G^A}=\overline{A^{-1}GA}=\bar{A^{-1}}\bar{G}\bar{A}$, where $\bar{A}$ (resp. $\bar{G}$) is obtained by deleting the third row and column in $A$ (resp. in the matrices of $G$).
Note also that the set $\cX(\fqs)$ of $\fqs$-rational points of $\cX$ has a short orbit $\cO$ under the action of $\aut(\cX)$, consisting of the set of $\fns$-rational points of the cone $Y^{n+1}=X^{n+1}-1$ in the plane $Z=0$ (see \cite[Section 3]{FG}). Hence, $\overline{G^A}$ acts naturally on $\bar{\cO}=\cH(\fns)$, where $\cH$ is the Hermitian curve with equation $Y^{n+1}=X^{n+1}-1$.

\section{A family of Galois subcovers of $\cX$}\label{SecEq}

In this section we provide
%MG
 equations and genera for a family of curves covered by the curve $\cX$ depending on three parameters.
% We show in Section 4, that for many choices of the parameters, the curves obtained are actually the fixed field by subgroups of $G$.
 Let $d_1,d_2,d_3$ be divisors of $n+1$, and consider the rational functions
 %MG
$$ u=x^\frac{n+1}{d_1},\,v=y^\frac{n+1}{d_2},\,w=z^\frac{n+1}{d_3} $$
in the function field $\mathbb{K}(x,y,z)$ of $\cX$. Then for the subfield $\mathbb{K}(u,v,w)$ we have the relations
\begin{equation}\label{uvw}
w^{d_3(n^2-n+1)}
=u^{d_1}(u^{d_1}-1)
\left(\frac{u^{d_1(n-1)}-1}{u^{d_1}-1}\right)^{n+1}, \qquad 
v^{d_2}=u^{d_1}-1.
\end{equation}
Let $L$ be the following subgroup of the group $G$ given in \eqref{groupG}:
$$L=\left\{(X,Y,Z,T)\mapsto(\gl^3b^nX,bY,\gl Z,T)\mid b^{n+1}=\gl^{n+1}=1 \right\}. $$

Clearly, $L$ has order $(n+1)^2$, and the fixed field $Fix(L)$ contains $x^{n+1}$, $y^{n+1}$ and $z^{n+1}$.
Actually, $Fix(L)$ coincides with $\KK(x^{n+1},y^{n+1},z^{n+1})$, since $\KK(x^{n+1},y^{n+1},z^{n+1})$ coincides with $\KK(x^{n+1},z^{n+1})$ and the degree of the extension $$\KK(x,y,z)/\KK(x^{n+1},z^{n+1})$$ is at most $(n+1)^2$. Then
$Fix(L)\subseteq\mathbb{K}(u,v,w)$ and we can consider the double extension of function fields
$$ Fix(L)\subseteq\mathbb{K}(u,v,w)\subseteq\mathbb{K}(x,y,z).$$

As $\mathbb{K}(x,y,z)/Fix(L)$ is a Galois extension, 
%MG
 $\mathbb{K}(x,y,z)/\mathbb{K}(u,v,w)$ is  Galois as well, that is, $\mathbb{K}(u,v,w)$ is the function field of the quotient curve of $\cX/H$
 %MG
 of $\cX$ with respect to some automorphism group $H\leq L$.

In order to provide irreducible equations for $\cX/H$, consider the rational function $\alpha\in\mathbb{K}(u,v)$ defined as
$$ \alpha = u^{d_1}(u^{d_1}-1)\left(\frac{u^{d_1(n-1)}-1}{u^{d_1}-1}\right)^{n+1}. $$
After some computation, we get the principal divisor of $\alpha$ in $\KK(u,v)$:
\begin{equation}\label{PrincipalDivisor}
\div(\alpha) = d_1 \sum_{i=1}^{d_2}Q_{0,i} + d_2 \sum_{i=1}^{d_1}Q_{\ga_i} + (n+1) \sum_{i=1}^{d_1(n-2)}\sum_{j=1}^{d_2}Q_{\gb_i,j} - \frac{d_1d_2n(n-1)}{(d_2,2d_1)}\sum_{i=1}^{(d_2,2d_1)}Q_{\infty,i},
\end{equation}
where $Q_{0,i}$ lies over the zero $P_0$ of $u$, $Q_{\ga_i}$ lies over the zero $P_{\ga_i}$ of $u^{d_1}-1$, $Q_{\gb_i,j}$ lies over the zero $P_{\gb_i}$ of $(u^{d_1(n-1)}-1)/(u^{d_1}-1)$, and $Q_{\infty,i}$ lies over the pole $P_\infty$ of $u$. Let
$$ D = \gcd\left(d_1,d_2,n+1,\frac{d_1d_2n(n-1)}{(d_2,2d_1)}\right), $$
$$ M = \gcd\left(D,d_3(n^2-n+1)\right) = \gcd\left(d_1,d_2,d_3(n^2-n+1)\right). $$

If $M=1$, then by equations \eqref{uvw}, $\mathbb{K}(u,v,w)/\mathbb{K}(u,v)$ is a Kummer extension of degree $d_3(n^2-n+1)$, and the quotient curve has irreducible equations
$$ 
\cX/H: \left\{\begin{array}{l}
W^{d_3(n^2-n+1)}=U^{d_1}V^{d_2}\left(\frac{U^{d_1(n-1)}-1}{U^{d_1}-1}\right)^{n+1}\\
V^{d_2}=U^{d_1}-1 \qquad\qquad\qquad\qquad\qquad\qquad\quad\;
\end{array} \right. .
$$
%MG

More generally, for $M\geq1$, both sides of the
%MG
 first equation in  \eqref{uvw} are a power of $M$, and we can factor the equation to obtain the irreducible curve
 %MG Capital letters
\begin{equation}\label{XsuH}
\cX/H: \left\{ \begin{array}{c}
W^{\frac{d_3}{M}(n^2-n+1)}=U^\frac{d_1}{M}V^{\frac{d_2}{M}}\left(\frac{U^{d_1(n-1)}-1}{U^{d_1}-1}\right)^{\frac{n+1}{M}} \\
V^{d_2}=U^{d_1}-1\qquad\qquad\qquad\qquad\qquad\quad\;
\end{array} \right. .
\end{equation}

\begin{remark}\label{rem7dic}
%MG
We are in a position to compute the order of the group $H$, that is, the degree of the extension $[\KK(x,y,z):\KK(u,v,w)]$.
By the Fundamental Equality (see \cite[Th. 3.1.11]{Sti}), the divisor of zeros of $x$ in $\KK(x,y,z)$ has degree $\deg(x)_0=[\KK(x,y,z):\KK(x)]=n^3+1$, and
$$ [\KK(x,y,z):\KK(u)]=\deg(x^\frac{n+1}{d_1})_0=\frac{(n+1)^2(n^2-n+1)}{d_1},\quad [\KK(u,v):\KK(u)]=d_2;  $$
hence,
$$ [\KK(x,y,z):\KK(u,v)]=\frac{(n+1)^2(n^2-n+1)}{d_1,d_2},\; [\KK(u,v,w):\KK(u,v)]=\frac{d_3(n^2-n+1)}{M}. $$
Therefore
$$ |H| =[\KK(x,y,z):\KK(u,v,w)]=\frac{M(n+1)^2}{d_1d_2d_3}. $$
\end{remark}

The general equations \eqref{XsuH} of $\cX/H$ have been obtained by working on $d_1,d_2,d_3(n^2-n+1)/M$. If we start from $d_1/M,d_2,d_3$, or from $d_1,d_2/M,d_3$, then we get irreducible equations for other quotient curves:
%\textcolor{blue}{
%these two equations below are such that the degree $[\KK(u,v,w):\KK(u,v)]= d_3(n^2-n+1)$ and $[\KK(u,v):\KK(u)]=d_2$ and $d_2/M$ respectively, different from the degree of the equations \eqref{XsuH}. The first one is like $M=1$ and the second one is not.
%} 
%\textcolor{red}{ok, partially: by repeating the same argument as in the remark, I think that the degree $[\KK(u,v):\KK(u)]=d_2$ is:
%$$ \frac{M(n+1)^2}{d_1d_2d_3}\quad\textrm{ for the curve \eqref{XsuH}}, $$
%$$ \frac{(n+1)^2}{d_1d_2d_3}\quad\textrm{ for the first curve below}, $$
%$$ \frac{M(n+1)^2}{d_1d_2d_3}\quad\textrm{ for the second curve below}. $$
%}
%\textcolor{blue}{I believe you are right, but it sounds strange anyway. So the modification in the next Theorem is ok too}

$$ \left\{ \begin{array}{c}
W^{d_3(n^2-n+1)}=U^\frac{d_1}{M}\left(U^{\frac{d_1}{M}(n-1)}-1\right)\left(\frac{U^{\frac{d_1}{M}(n-1)}-1}{U^{\frac{d_1}{M}}-1}\right)^n \\
V^{d_2}=U^{\frac{d_1}{M}}-1 \qquad\qquad\qquad\qquad\qquad\qquad\qquad\quad\;
\end{array} \right. ,$$
$$ \left\{ \begin{array}{c}
W^{d_3(n^2-n+1)}=U^{d_1}\left(U^{d_1(n-1)}-1\right)\left(\frac{U^{d_1(n-1)}-1}{U^{d_1}-1}\right)^n \\
V^{\frac{d_2}{M}}=U^{d_1}-1 \qquad\qquad\qquad\qquad\qquad\qquad\quad\quad\;
\end{array} \right. .$$

%Over the function fields of these curves we can also consider the rational functions $$u,%\qquad :v:w^{\frac{n^2-n+1}{e}}:1),$$ 
%MG if we want to talk about morphism we should specify domain, codomain, etc
For any divisor $e$ of $n^2-n+1$ let $s=w^{\frac{n^2-n+1}{e}}$; then $\KK(u,v,s)$ is the function field of new subcovers of $\cX$. The degree of the covering can be easily computed also for these subcovers, arguing as in Remark \ref{rem7dic}.

To sum up, the following result is obtained.

\begin{theorem}\label{equations}
Let $d_1$, $d_2$, and $d_3$ be divisors of $n+1$,  and let $e$ be a divisor of $n^2-n+1$. For
$$
M=\gcd\left(d_1,d_2,d_3(n^2-n+1)\right),
$$
the following equations define $\mathbb F_{n^6}$-maximal curves which are Galois subcovers of $\cX$:
\begin{equation}\label{C1}
 \cC_1: \left\{ \begin{array}{c}
S^{\frac{d_3}{M}e}=U^\frac{d_1}{M}V^{\frac{d_2}{M}}\left(\frac{U^{d_1(n-1)}-1}{U^{d_1}-1}\right)^{\frac{n+1}{M}} \\
V^{d_2}=U^{d_1}-1 \qquad\qquad\qquad\quad\;\:
\end{array} \right. ,
\end{equation}
\begin{equation}\label{C2}
 \cC_2: \left\{ \begin{array}{c}
S^{d_3e}=U^\frac{d_1}{M}\left(U^{\frac{d_1}{M}(n-1)}-1\right)\left(\frac{U^{\frac{d_1}{M}(n-1)}-1}{U^{\frac{d_1}{M}}-1}\right)^n \\
V^{d_2}=U^{\frac{d_1}{M}}-1 \qquad\qquad\qquad\qquad\qquad\quad\;\:
\end{array} \right. ,
\end{equation}
\begin{equation}\label{C3}
 \cC_3: \left\{ \begin{array}{c}
S^{d_3e}=U^{d_1}\left(U^{d_1(n-1)}-1\right)\left(\frac{U^{d_1(n-1)}-1}{U^{d_1}-1}\right)^n \\
V^{\frac{d_2}{M}}=U^{d_1}-1 \qquad\qquad\qquad\qquad\quad\quad\;\;
\end{array} \right. .
\end{equation}
The degree of the covering is $\frac{(n^2-n+1)M(n+1)^2}{ed_1d_2d_3}$ for $\cC_1$ and $\cC_3$, and $\frac{(n^2-n+1)(n+1)^2}{ed_1d_2d_3}$ for $\cC_2$.
\end{theorem}

Note that, when $\frac{(n^2-n+1)M(n+1)^2}{ed_1d_2d_3}=1$
or $\frac{(n^2-n+1)(n+1)^2}{ed_1d_2d_3}=1$, this theorem provides models for the
%MG
 GK-curve; in some cases they are plane models.
%\textcolor{blue}{
%When $\frac{eM(n+1)^2}{d_1d_2d_3}=1$, this result provide equations for the GK-curve. Choosing $d_1=1, d_2=d_3=n+1, e=e=n^2-n+1$, in the curve $\cC_1$,
%the plane equation becomes
%$$
%S^{n^3+1}=(V^{n+1}+1)\left(\frac{((V^{n+1}+1)^{(n-1)}-1)^{n+1}}{V^{(n^2+n)}}\right).
%$$
%Does this curve have the same genus of the GK now? Yes, using the genus formula in Theorem from section Genus.
%}

Now we compute the genera of the curves described in Theorem \ref{equations} for $e=n^2-n+1$, i.e. $s=w$. This is done via Kummer theory.

%For the curves $\cC_2$ and $\cC_3$, we calculate the ramification indices in the double Kummer extension
%$$ \KK(u)\subseteq\KK(u,v)\subseteq\KK(u,v,w). $$
%For the curve $\cC_1$, we also need to consider the double Kummer extension
%$$ \KK(v)\subseteq\KK(u,v)\subseteq\KK(u,v,w). $$

%Then we apply the Riemann-Hurwitz genus formula and obtain the following result
%, whose detailed proof is omitted since it goes beyond the scope of
%the present extended abstract
%.

\begin{theorem}\label{genera}
Let $e=n^2-n+1$. Then the genera of the curves $\cC_1$, $\cC_2$, and $\cC_3$ described in Theorem {\rm \ref{equations}} are the following:
\begin{small}
\begin{equation}\label{gC1}
g(\cC_1) = 1 + \frac{1}{2}\Big[d_1d_2\frac{d_3(n^2-n+1)}{M}(n-1)-d_2(\frac{d_1}{M},\frac{d_3(n^2-n+1)}{M})-d_1(\frac{d_2}{M},\frac{d_3(n^2-n+1)}{M})+
\end{equation}
$$-d_1d_2(n-2)(\frac{d_3(n^2-n+1)}{M},\frac{n+1}{M})-\left((d_1,d_2)\frac{d_3(n^2-n+1)}{M},\frac{2d_1d_2}{M}\right)\Big]$$
\end{small}
and, for $i=2,3$,
\begin{small}
\begin{equation}\label{gCi}
g(\cC_i) = 1 + \frac{1}{2}\left[hkr(n-1)-k(h,r)-h(k,r)-hk(n-2)(r,n+1)-\left((h,k)r,2hk\right)\right],
\end{equation}
\end{small}
where
\begin{small}
$$ r=d_3(n^2-n+1),\quad
h= \left\{ \begin{array}{c} d_1/M\;\,\textrm{for}\;\cC_2 \\ d_1\qquad\textrm{for}\;\cC_3 \end{array} \right. ,\quad
k= \left\{ \begin{array}{c} d_2\qquad\textrm{for}\;\cC_2 \\ d_2/M\;\,\textrm{for}\;\cC_3 \end{array} \right. .
$$
\end{small}
\end{theorem}

\begin{proof}
We start with $\cC_1$, and use the notation of equation \eqref{PrincipalDivisor} for the zeros and poles of
$$\ga = u^\frac{d_1}{M}v^{\frac{d_2}{M}}\left(\frac{u^{d_1(n-1)}-1}{u^{d_1}-1}\right)^{\frac{n+1}{M}} \in\KK(u,v).$$

By \cite[Prop. 3.7.3]{Sti} we compute the ramification indices in the Kummer extension $\KK(u,v)/\KK(u)$ of degree $d_2$: 
$$ e(Q_{\ga_i}\mid P_{\ga_i})=d_2, e(Q_{0,i}\mid P_0)=1, e(Q_{\gb_i,j}\mid P_{\gb_i})=1, e(Q_{\infty,i}\mid P_\infty)=\frac{d_2}{\gcd(d_1,d_2)}. $$

The Riemann-Hurwitz formula applied to the extension $\KK(u,v)/\KK(u)$ yields
$$ g(\KK(u,v)) = 1+\frac{1}{2}(d_1d_2-d_1-d_2-\gcd(d_1,d_2)). $$

Let $\bar{P}_0$ be the zero and $\bar{P}_\infty$ the pole of $v$ in $\KK(v)$. Then, in the extension $\KK(u,v)/\KK(v)$, the places lying over $\bar{P}_0$ are $Q_{\ga_1},\ldots,Q_{\ga_{d_1}}$, with ramification index $1$; the places over $\bar{P}_\infty$ are $Q_{\infty,1},\ldots,Q_{\infty,(d_1,d_2)}$, with ramification index $d_1/\gcd(d_1,d_2)$.

We compute the ramification indices in the Kummer extension $\KK(u,v,s)/\KK(u,v)$ of degree $\frac{d_3}{M}(n^2-n+1)$. We have
$$ v_{Q_{\ga_i}}(\ga) = e(Q_{\ga_i}\mid P_{\ga_i})\cdot v_{P_{\ga_i}}\left(u^\frac{d_1}{M}\left(\frac{u^{d_1(n-1)}-1}{u^{d_1}-1}\right)^{\frac{n+1}{M}}\right) + e(Q_{\ga_i}\mid \bar{P}_0)\cdot v_{\bar{P}_0}\left(v^{\frac{d_2}{M}}\right) = \frac{d_2}{M}, $$
hence
$$e(R_{\ga_i,j}\mid Q_{\ga_i})=\frac{\frac{d_3}{M}(n^2-n+1)}{\gcd\left(\frac{d_3}{M}(n^2-n+1),\frac{d_2}{M}\right)},$$
where $R_{\ga_i,j}$ is a place of $\KK(u,v,s)$ lying over $Q_{\ga_i}$. The theory of Kummer extensions also gives the ramification indices
$$\frac{\frac{d_3}{M}(n^2-n+1)}{\gcd\left(\frac{d_3}{M}(n^2-n+1),v_{Q}(\ga)\right)}$$
of the places of $\KK(u,v,s)$ lying over $Q$, for all places $Q$ of $\KK(u,v)$. Then the different divisor of $\KK(u,v,s)/\KK(u,v)$ has degree
\begin{small}
$$\deg(\textrm{Diff})= d_1\left(m-\left(m,\frac{d_2}{M}\right)\right)+d_2\left(m-\left(m,\frac{d_1}{M}\right)\right)+\qquad\qquad\qquad\qquad$$
$$\qquad\qquad+d_1(n-2)d_2\left(m-\left(m,\frac{n+1}{M}\right)\right)+(d_1,d_2)\left(m-\left(m,\frac{d_1d_2(n^2-n)}{M(d_1,d_2)}\right)\right) = $$

$$ = d_1\left(m-\left(m,\frac{d_2}{M}\right)\right)+d_2\left(m-\left(m,\frac{d_1}{M}\right)\right) + \qquad\qquad\qquad $$
$$ \qquad\qquad + d_1(n-2)d_2\left(m-\left(m,\frac{n+1}{M}\right)\right)+(d_1,d_2)\left(m-\left(m,\frac{2d_1d_2}{M(d_1,d_2)}\right)\right), $$
\end{small}
where $m=d_3(n^2-n+1)/M$. Finally, the Riemann-Hurwitz formula applied to the extension $\KK(u,v,s)/\KK(u,v)$ provides the genus of $\cC_1$.

The curves $\cC_2$ and $\cC_3$ are both defined by equations of the form
$$
\cC_i: \left\{ \begin{array}{c}
S^{r}=U^{a}\left(U^{a(n-1)}-1\right)\left(\frac{U^{a(n-1)}-1}{U^{a}-1}\right)^n \\
V^{b}=U^{a}-1 \qquad\qquad\qquad\qquad\qquad\quad
\end{array} \right. .
$$

The genus of $\KK(u,v)$ is obtained as above:
$$ g(\KK(u,v)) = 1+\frac{1}{2}(ab-a-b-\gcd(a,b)). $$

Similar computations yield the degree of the different divisor of the Kummer extension $\KK(u,v,s)/\KK(u,v)$:
$$\deg(\textrm{Diff})= a\left(r-\gcd\left(r,b\right)\right)+b\left(r-\gcd\left(r,a\right)\right)+\qquad\qquad\qquad\qquad\qquad\qquad$$
$$\qquad\qquad+a(n-2)b\left(r-\gcd\left(r,n+1\right)\right)+\gcd(a,b)\left(r-\gcd\left(r,\frac{2ab}{\gcd(a,b)}\right)\right), $$
and the Riemann-Hurwitz formula applied to the extension $\KK(u,v,s)/\KK(u,v)$ provides the genus of $\cC_i$, $i=2,3$.

\end{proof}

\begin{remark}\label{61}
The previous results provide new equations of $\fqs$-maximal curves for many genera. To exemplify this, consider the case $n=5$. Then Theorem {\rm \ref{equations}} provides new equations for the following genera:
$$
\begin{array}{c}
37, 74, 109, 121, 148, 220, 242, 361, 442, 484, 724, 1450,  \\
160, 233, 469, 478, 496, 737, 1477, 1486  \\
\end{array}
$$
Up to our knowledge, the integers in the second row are new values in the spectrum of genera of $\FF_{5^6}$-maximal curves.
\end{remark}
%$$ \begin{array}{1} 74,76,148,220,241,242,442,484,724, \\ 26,36,52,72,120,146,154,240. $$ \end{array}

\section{The Galois groups}\label{quattro}
In this section we always assume $$ \gcd\left(d_1,d_2,d_3(n^2-n+1)\right)=1.$$
In some cases we are able to give an explicit description of the automorphism groups $H$ of order $\frac{(n+1)^2}{d_1d_2d_3}$ such that $Fix(H)=\mathbb K(u,v,w)$.
We also provide an alternative computation of the genus of $\cX/H$, by means of the Riemann-Hurwitz genus formula and \cite[Prop. 3.2]{FG}. 

We 
%MG
state  Proposition 3.2 from \cite{FG} in a slightly different form: in the original paper  the authors consider the model 
%MG
$\cC$ of the GK curve lying on the cone over the Hermitian curve $\cK$ with equation $Y^n+Y = X^{n+1}$;
it is not difficult to see that the same computations hold for the curve $\cX$ lying on the cone over the Hermitian curve $\cH: Y^{n+1}=X^{n+1}-1$. 
This relies on the fact that $\cC$ and $\cX$ are projectively equivalent, with a projectivity defined over $\fns$ which maps the Hermitian cone over $\cK$ to the Hermitian cone over $\cH$.

\begin{proposition}\label{genusFG}\cite[Prop. 3.2]{FG}
Let $L$ be a tame subgroup of $Aut(\cX)$, $\bar{L}$ the projection of $L$ to $PGU(n,3)$ and $L_\Lambda=L \cap \Lambda$, where $\Lambda$ is defined in equation \eqref{Lambda}. Assume that no non-trivial element in $\bar{L}$ fixes a point in $\cH\setminus \cH(\fns)$, where $\cH$ is the Hermitian curve $Y^{n+1}=X^{n+1}-1$. Then:
\begin{small}
$$g_L= g_{\bar{L}}+\frac{(n^3+1)(n^2-|L_\Lambda|-1)-|L_\Lambda|(n^2-n-2)}{2|L|},$$
\end{small} 
where $g_L$ is the genus of the quotient curve $\cH/\bar{L}$.
\end{proposition} 

\begin{case}\label{Hcase1}
Suppose that $d_1\mid 3d_3$ with $\gcd(d_1,d_2)=1$. Then $\KK(u,v,w)$ is the quotient curve of $\cX$ with respect to the group
$$H=\left\{(X,Y,Z,T)\mapsto(\gl^3b^nX,bY,\gl Z,T)\mid b^{\frac{n+1}{d_1d_2}}=\gl^{\frac{n+1}{d_3}}=1 \right\}.$$

In fact, by Remark \ref{rem7dic}, the size $\frac{(n+1)^2}{d_1d_2d_3}$ of $H$ coincides with the degree of the extension 
$$\KK(x,y,z)/\KK(u,v,w).$$ Also, $u$, $v$, and $w$ are all fixed by $H$ since
$$
\lambda^{(n+1)/d_3}=1,\qquad b^{(n+1)/d_2}=(b^{(n+1)/d_2d_1})^{d_1}=1,
$$
and
$$
(\gl^3b^n)^{(n+1)/d_1}=(\gl^{(n+1)/d_3})^{3d_3/d_1} b^{-((n+1)/d_1)}=(b^{-((n+1)/d_1d_2)})^{d_2}=1.
$$
The projection $\bar{H}$ of $H$ on $\textrm{PGU}(3,n)$ is
$$\bar{H}=\left\{[\gl^3b^n,b,1] \mid b^{\frac{n+1}{d_2d_1}}=\gl^{\frac{n+1}{d_3}}=1 \right\},$$
$$ \textrm{with}\quad |\bar{H}|=\frac{(n+1)^2}{d_1d_2d_3\gcd\left(3,\frac{n+1}{d_3}\right)}, $$
where $[\gl^3b^n,b,1]$ denotes the automorphism $(X,Y,T)\mapsto(\gl^3b^nX,bY,T)$.
No non-trivial element in $\bar{H}$ fixes a point in $\cH\setminus\cH(\fns)$, and 
$$H_{\Lambda}=\left\{\left[1,1,\gl,1\right]\mid \gl^{n^2-n+1}=\gl^{\frac{n+1}{d_3}}=1 \right\}$$
has size $\gcd\left(\frac{n+1}{d_3},3\right)$. Then by Proposition \ref{genusFG} the genus of $\cX/H$ is
\begin{small}
$$ g_{H} = g_{\bar{H}} + \frac{d_1d_2d_3[(n^3+1)(n^2-\gcd(3,(n+1)/d_3)-1)-\gcd(3,(n+1)/d_3)(n^2-n-2)]}{2(n+1)^2}, $$
\end{small}
where $g_{\bar{H}}$ is the genus of $\cH/\bar{H}$.
The only points of $\cH$ that can be fixed by a non-trivial element 
%$h_{\gl,b}\in
in $
\bar{H}$ are the points on the fundamental frame. It is easily seen that
\begin{itemize}
\item[(i)] $[\gl^3b^n,b,1]$ fixes $P_i=(0,\ga_i,1)$, $i=1,\ldots,n+1$, if and only if $b=1$;
\item[(ii)] $[\gl^3b^n,b,1]$ fixes $Q_j=(\gb_j,0,1)$, $j=1,\ldots,n+1$, if and only if $\lambda^3=b$;
\item[(iii)] $[\gl^3b^n,b,1]$ fixes $R_k=(\gb_k,1,0)$, $k=1,\ldots,n+1$, if and only if $\lambda^3=b^2$.
\end{itemize}

Let $\bar{H}_P$ denote the stabilizer of $P$ in $\bar{H}$. We distinguish two cases.

({\bf A}) $3$ does not divide $(n+1)/d_3$. Then $\,\gl\mapsto\gl^3\,$ is an automorphism of the multiplicative group of the $((n+1)/d_3)$-th roots of unity.

\begin{itemize}

\item[(i)]
We have
$$ \bar{H}_{P_i}=\left\{[\gl^3,1,1] \mid \gl^{\frac{n+1}{d_3}}=1\right\},\quad\textrm{hence}\quad |\bar{H}_{P_i}|= \frac{n+1}{d_3}. $$

\item[(ii)]
We have
$$ \bar{H}_{Q_j}=\left\{[1,b,1] \mid b^{\frac{n+1}{d_2d_1}}=1,b=\gl^3\textrm{ for some }\gl\textrm{ with } \gl^{\frac{n+1}{d_3}}=1\right\},$$hence
$$|\bar{H}_{Q_j}|= \gcd\left(\frac{n+1}{d_3},\frac{n+1}{d_1d_2}\right). $$

\item[(iii)] We distinguish two subcases.

\begin{itemize}

\item $\frac{n+1}{d_1d_2}$ is even. Then
$$ \bar{H}_{R_k}=\left\{[b,b,1] \mid (b^2)^{\frac{n+1}{d_3}}=1,(b^2)^{\frac{n+1}{2d_1d_2}}=1 \right\} \quad\textrm{and}\quad b\neq-b, $$
hence
$$ |\bar{H}_{R_k}|=2\gcd\left(\frac{n+1}{d_3},\frac{n+1}{2d_1d_2}\right)=\gcd\left(\frac{2(n+1)}{d_3},\frac{n+1}{d_1d_2}\right). $$

\item $\frac{n+1}{d_1d_2}$ is odd. Then $\,b\mapsto b^2\,$ is an automorphism of the multiplicative group of the $((n+1)/d_1d_2)$-th roots of unity, and
$$ \bar{H}_{R_k}=\left\{[b,b,1] \mid \gl^{\frac{n+1}{d_3}}=b^{\frac{n+1}{d_1d_2}}=1,\gl^3=b^2 \right\}; $$
hence,
$$ |\bar{H}_{R_k}|=\gcd\left(\frac{n+1}{d_3},\frac{n+1}{d_1d_2}\right)=\gcd\left(\frac{2(n+1)}{d_3},\frac{n+1}{d_1d_2}\right). $$

\end{itemize}

\end{itemize}

Therefore, if $3\nmid(n+1)/d_3$ then the Hurwitz formula applied to the covering $\cH\rightarrow\cH/\bar H$ provides the genus of $\cH/\bar H$:
\begin{small}
$$
g_{\bar{H}}=1+\frac{1}{2\frac{(n+1)^2}{d_1d_2d_3}}\left[n^2-n-2-(n+1)\left(\frac{n+1}{d_3}+\gcd(\frac{n+1}{d_1d_2},\frac{n+1}{d_3})+\gcd(\frac{n+1}{d_1d_2},\frac{2(n+1)}{d_3})-3\right)\right],
$$
\end{small}
that is,
\begin{small}
$$
g_{\bar{H}}=1+\frac{d_1d_2d_3}{2(n+1)}\Big[n-2-\frac{n+1}{d_3}-\gcd(\frac{n+1}{d_1d_2},\frac{n+1}{d_3})-\gcd(\frac{n+1}{d_1d_2},\frac{2(n+1)}{d_3})+3\Big],
$$
\end{small}
hence
\begin{small}
$$ g_{H} = 1+\frac{d_1d_2d_3}{2(n+1)}\left(n+1-\frac{n+1}{d_3}-\gcd(\frac{n+1}{d_1d_2},\frac{n+1}{d_3})-\gcd(\frac{n+1}{d_1d_2},\frac{2(n+1)}{d_3})\right)+$$
\begin{equation}\label{GGG}
+\frac{d_1d_2d_3\left(n^3-2n^2+n\right)}{2}.
\end{equation}
\end{small}

({\bf B}) $3$ divides $(n+1)/d_3$. Let $\gl'=\gl^3$, then
$$ \bar{H}=\left\{[\gl'b^n,b,1] \mid (\gl')^{\frac{n+1}{3d_3}}=b^{\frac{n+1}{d_1d_2}}=1 \right\}. $$

The same arguments yield 

%yy:=(n+1) div (d1*d2);
%K:=((n^3+1)*(n^2-4)-3*(n^2-n-2));
%card:=((n+1)^2) div (d1*d2*d3);
%K1:=K div (2*card);
%A:=(n+1) div (3*d3);
%B:=Gcd(A,yy);
%C:=Gcd(2*A,yy);
%ram:=(n+1)*(A+B+C-3);
%gbarH':=1+ ( (n^2-n-2-ram) div ((2*card) div 3) );
%gH':=gbarH'+K1;

\begin{small}
$$ g_{H} = 1+\frac{3d_1d_2d_3}{2(n+1)}\left(n+1-\frac{n+1}{3d_3}-\gcd(\frac{n+1}{d_1d_2},\frac{n+1}{3d_3})-\gcd(\frac{n+1}{d_1d_2},\frac{2(n+1)}{3d_3})\right) + $$
$$ + \frac{d_1d_2d_3\left(n^3-2n^2-n+2\right)}{2}. $$
\end{small}

\end{case}

\begin{case}\label{Hcase2}
Suppose  $d_1\mid d_2$,  and $(d_1,d_3(n^2-n+1))=1$. Then 
 $\KK(u,v,w)$ is the quotient curve of $\cX$ with respect to the group
$$ H=\left\{(X,Y,Z,T)\mapsto(\gl^3b^nX,bY,\gl Z,T)\mid b^{\frac{n+1}{d_2}}=\gl^{\frac{n+1}{d_1d_3}}=1 \right\}. $$
This follows from 
$$
\lambda^{(n+1)/d_3}=(\lambda^{(n+1)/d_1d_3})^{d_1}=1,\qquad b^{(n+1)/d_2}=1,
$$
and
$$
(\gl^3b^n)^{(n+1)/d_1}=(\gl^{\frac{n+1}{d_1d_3}})^{3d_3}b^{-((n+1)/d_1)}=(b^{-((n+1)/d_2)})^{d_2/d_1}=1.
$$
Similar computations provide the genus of $\cX/H$:
\begin{small}
$$ g_{H} = 1+\frac{d_1d_2d_3m}{2(n+1)}\left(n+1-\frac{n+1}{d_1d_3m}-\gcd(\frac{n+1}{d_2},\frac{n+1}{d_1d_3m})-\gcd(\frac{n+1}{d_2},\frac{2(n+1)}{d_1d_3m})\right) + $$
$$ + \frac{d_1d_2d_3\left[n^3-2n^2+(2-m)n+m-1 \right]}{2}, $$
\end{small}
where $m=\gcd(3,(n+1)/(d_1d_3))$.

\end{case}

\section{Another family of Galois subcovers of $\cX$}

In this section we consider another subgroup of the group $G$ given in \eqref{groupG}.
Let $c \mid(n+1)$, $d\mid(n^2-n+1)$, and consider the following automorphism group $K$ of $\cX$ of size $(n^3+1)/(cd)$:
$$ K=\left\{(X,Y,Z,T)\mapsto(b^{-1}X,bY,\gl Z,T)\mid b^\frac{n+1}{c}=1, \gl^\frac{n^2-n+1}{d}=1 \right\}. $$
%MG
Consider the following rational functions
$$ u=x^{\frac{n+1}{c}},\quad v=xy,\quad w=z^{\frac{n^2-n+1}{d}} $$
in the function field $\KK(x,y,z)$ of $\cX$; then the following relations hold:
\begin{equation}\label{K'}
w^d = v\left(1+u^c+u^{2c}+\ldots+u^{(n-2)c}\right),\quad v^{n+1}=u^{2c}-u^c. \end{equation}
In the double field extension
$ \KK(u,v,w)\subseteq Fix(K)\subseteq\KK(x,y,z) $ 
we have
$$ [\KK(x,y,z):\KK(u,v,w)]\leq\frac{n^3+1}{cd}=[\KK(x,y,z):Fix(K)], $$
which implies $Fix(K)=\KK(u,v,w)$.

%MG
Equations \eqref{K'} define an irreducible curve. To show this, let $P=(0,a)$ be an affine point of the Hermitian plane curve $\cH:Y^{n+1}=X^{n+1}-1$, and let $\bar{P}$ be a place of the curve $\cW:V^{n+1}=U^{2c}-U^c$ centered at the image $\varphi(P)$ of $P$ under the rational map
$$ \varphi:\cH\rightarrow \cW,\quad \varphi(X,Y,T)=(X^{\frac{n+1}{c}},XY,T) .$$
The rational function $\gb=xy(1+x^{n+1}+x^{2(n+1)}\ldots+x^{(n-2)(n+1)})\in\KK(x,y)$ has valuation $v_P(\gb)=1$ at $P$, hence the pull-back $\ga=v(1+u^c+\ldots+u^{(n-2)c})\in\KK(u,v)$ of $\gb$ has valuation $v_{\bar P}(\ga)=1$ at $\bar P$, since $v_P(\gb)=e(P\mid\bar{P})\cdot v_{\bar P}(\ga)$.
Hence 
%MG
%the equations \eqref{K'} are irreducible, i.e. 
the quotient curve $\cX/K$ has irreducible equations
\begin{equation}\label{XsuK}
\cX/K:\left\{
\begin{array}{c}
W^d = V\left(1+U^c+U^{2c}+\ldots+U^{(n-2)c}\right) \\
V^{n+1}=U^{2c}-U^c \qquad\qquad\qquad\qquad\qquad\quad
\end{array}\right. .
\end{equation}
By the Hurwitz formula applied to the tame covering $\cX\rightarrow\cX/K$, it is easy to check that the genus of $\cX/K$ is
\begin{equation}\label{gXsuK}
 g(\cX/K) = \frac{c}{2}\left[(d-1)n^2+n-d-\gcd\left(2,\frac{n+1}{c}\right)\right]+1.
 \end{equation}
%(The same value is obtained by applying the proposition 3.2 in Fanali-Giulietti as we did above for $H'$).

\section{New examples of maximal curves not (Galois) covered by the Hermitian curve}\label{62}

Let a curve $\cY$ be a subcover of the Hermitian curve $\cH$ by the $\fqs$-rational map
$$ \varphi: \cH \rightarrow \cY. $$
Then for the degree $\deg(\varphi)$ we have the following bounds:
$$ \frac{\cH(\fqs)}{\cY(\fqs)}\leq \deg(\varphi) \leq \frac{2g(\cH)-2}{2g(\cY)-2}. $$
In particular, the lower bound $L_{\cH,\cY}=\cH(\fqs)/\cY(\fqs)$ and the upper bound $U_{\cH,\cY}=(2g(\cH)-2)/(2g(\cY)-2)$ satisfy $\left\lceil{L_{\cH,\cY}}\right\rceil \leq \left\lfloor{U_{\cH,\cY}}\right\rfloor $.

Therefore, a curve $\cY$ having $\left\lceil{L_{\cH,\cY}}\right\rceil > \left\lfloor{U_{\cH,\cY}}\right\rfloor $ cannot be a subcover of the Hermitian curve. By applying this argument to the curves given in Theorems \ref{equations} and \ref{genera}, we get many new examples of curves which are not covered by the Hermitian curve.

To exemplify this, we list in the table below some genera of curves not covered by the Hermitian curve. We remark that for such curves we have both the genus and explicit equations.
%MG the Galois group of the extension is not the automorphism group of the curve.
%; for some of them, we also have an explicit description of the automorphism group.

\begin{small}
\begin{center}
\begin{table}[htbp]
\caption{New maximal curves not covered by the Hermitian curve}\label{tabella}
\begin{tabular}{|c|c|c|c|}
\hline $g$ & $n$ & $(d_1,d_2,d_3)$ & $\textrm{Reference}$ \\
\hline & & (1,18,6), (2,9,6), (2,18,3), (2,18,6), (3,18,6), & \\
 233416 & 17 & (6,9,6), (6,18,3), (6,18,6), (9,2,6), (9,6,6), & Thm. \ref{genera} \eqref{gC1},\eqref{gCi} \\
 & & (9,18,2), (9,18,6), (18,1,6), (18,2,3), (18,2,6), & \\
 & & (18,3,6), (18,6,3), (18,6,6), (18,9,2), (18,9,6) & \\
\hline 233398 & 17 & (9,18,2) & Thm. \ref{genera} \eqref{gCi} \\
\hline & & (1,24,8), (8,3,8), (24,8,1), (24,1,8), (2,24,8),  & \\
1064701 & 23 & (3,8,8), (3,24,8), (4,24,8), (6,8,8), (6,24,8), & Thm. \ref{genera} \eqref{gC1},\eqref{gCi} \\
 & & (8,3,8), (8,6,8), (8,12,8), (8,24,1), (8,24,2) & \\
\hline 1064689 & 23 & (2,24,8), (4,24,8), (6,8,8), & Thm. \ref{genera} \eqref{gCi} \\
 & &  (6,24,8), (8,6,8), (8,12,8) & \\
\hline 3206257 & 23 & (2,24,24), (4,24,24), (6,24,24), (8,6,24), (8,12,24) & Thm. \ref{genera} \eqref{gCi} \\
\hline 3402406 & 29 & (30,10,1), (10,30,1), (10,15,2), & Thm \ref{genera} \eqref{gC1} \\
 & & (30,2,5), (10,6,5), (10,3,10) & \\
\hline 5570731 & 32 & (33,11,1), (11,33,1), (11,3,11) & Thm \ref{genera} \eqref{gC1} \\
\hline
\end{tabular}
\end{table}
\end{center}
\end{small}

\begin{remark}
Let $\cH$ be the Hermitian curve over $\fqs$, and $\cY$ an $\fqs$-maximal curve of genus $g$ which is $\fqs$-covered by $\cH$. If $g>f(q)$, where
$$ f(q) = \frac{\sqrt{q^5+2q^4+q^3+q^2+2q+1}-q^2-1}{2q}, $$
then the degree of the covering $\cH\rightarrow\cY$ is uniquely determined as the unique integer $d$ such that
$$ L_{\cH,\cY} \leq d \leq U_{\cH,\cY}. $$
\end{remark}
\begin{proof}
By direct computation, $g>f(q)$ is equivalent to $U_{\cH,\cY} - L_{\cH,\cY} < 1$, which implies $\left\lceil{L_{\cH,\cY}}\right\rceil = \left\lfloor{U_{\cH,\cY}}\right\rfloor $.
\end{proof}

\begin{theorem}\label{NotGaloisCovered}
Let $n\geq7$ be a power of a prime $p$, and $k$ a divisor of $n+1$ with $k<\sqrt{n+1}+1$. Define $d_1=(n+1)/k$, $d_2=1$, and $d_3=n+1$. 
Then the curve $\cC_1$ given in Theorem {\rm \ref{equations}} is not Galois covered by the Hermitian curve $\cH$ over $\FF_{n^6}$.
\end{theorem}

\begin{proof}
Let $\cH$ be given in the form
$$ \cH: Y^{n^3+1} = X^{n^3+1}+X, $$
and denote by $P_\infty$ the point at infinity of $\cH$. 
Suppose $\cC_1$ is Galois covered by $\cH$;
%MG
 then $\cC_1$ is isomorphic to the quotient curve $\cH/N$ for some subgroup $N$ of $\aut(\cH)$.

The genus of $\cC_1$ can be computed by equation \eqref{GGG}, whence $L_{\cH,\cC_1}>kn-1$ if and only if
\begin{small}
$$ n^8-k(k-2)n^7-2n^6+n^5-(k-1)[2k+1-(k,2)]n^4+[2k-1-(k,2)]n^3-k^2n+2k>0, $$
\end{small}
while $U_{\cH,\cC_1}<kn+1$ if and only if
\begin{small}
$$ n^5-2kn^4+2(k-1)n^3-[k(k,2)-k-1]n^2-[(k,2)(k+1)+k-1]n+2k-(k,2)-1>0. $$
\end{small}
For $n\geq7$, both conditions are implied by the hypothesis $k<\sqrt{n+1}+1$.
Then $|N|=kn$.

Let $S$ be a Sylow $p$-subgroup of $N$. $S$ has a fixed $\FF_{n^6}$-rational point $P\in\cH$, since $S$ acts on $\cH(\FF_{n^6})$ and $|\cH(\FF_{n^6})|\equiv1(\mod\, p)$; as all Sylow $p$-subgroups of $\aut(\cH)$ are conjugate, then we assume w.l.o.g. that $S$ fixes $P_\infty$. Moreover, the action of $S$ on $\cH(\FF_{n^6})\setminus\left\{P_\infty\right\}$ is semiregular, i.e. each element of $S$ has no fixed point but $P_\infty$; %this holds more generally for zero p-rank curves, see \cite[Lemma 11.129]{HKT}.
hence the orbit $\cO$ of $P_\infty$ under $N$ satisfies $|\cO|\equiv1(\mod\,n)$.

Suppose $P_\infty$ is not fixed by $N$, then $|\cO|\geq n+1$. Hence, by the orbit-stabilizer theorem, $n$ divides the cardinality of the stabilizer $N_Q$ of $Q$ in $N$, for all $Q\in\cO$; then a Sylow $p$-subgroup $M_Q$ of $N_Q$ has size $n$. $M_Q$ and $M_R$ have trivial intersection for $Q\neq R$ in $\cO$, by the semiregularity of $S$. Therefore $N$ has at least $1+(n+1)(n-1)=n^2$ elements, and $k\geq n$, against the hypothesis.

%MG
Therefore the whole $N$ fixes $P_\infty$. If $k=1$, then $\cC_1\cong\cX$, the GK curve, and the thesis holds; otherwise, the genus of $\cH/N$ can be computed by \cite[Th. 4.4]{GSX}:
$$ g(\cH/N) = \frac{n^3-p^w}{2kn}\left(n^3-(k-1)p^v\right) = \frac{p^{5u}-p^{3u-v}-(k-1)p^{3u-w}+k-1}{2k}, $$
where $n=p^u$ and $v,w$ are non-negative integers satisfying $u=v+w$.

%MG
On the other hand, the genus of $\cC_1$ given in equation \eqref{GGG} is
$$ g(\cC_1) = \frac{n^5-2n^3+n^2+2k-1-h}{2k},\quad\textrm{where}\quad h=\left\{ \begin{array}{c}
n+2\quad\textrm{if}\;k\;\textrm{is even}  \\
1\qquad\;\;\textrm{if}\;k\;\textrm{is odd}
\end{array} \right. . $$
Hence $g(\cH/N)=g(\cC_1)$ reads
$$ k=\frac{2p^{3u}+p^{3u-w}-p^{3u-v}-p^{2u}+h}{p^{3u-w}+1}.$$

We have the following possibilities for $v$ and $w$: either $v=0$ and $w=u$, or $v\leq u/2$ and $w\geq u/2$, or $v>u/2$ and $w<u/2$. By considering separately each case, it is 
%MG
easily shown after some computation that
$$ \left(p^{3u-w}+1\right) \; \nmid \; \left(2p^{3u}+p^{3u-w}-p^{3u-v}-p^{2u}+h\right), $$
against the fact that $k$ is integer.
\end{proof}

\begin{theorem}\label{NotGaloisCoveredII}
Let $n>3$ be a prime power, $k$ a divisor of $n+1$ such that $3\nmid(n+1)/k$ and $k<\sqrt{n+1}+1$; if $3\mid(n+1)$, assume also $n\geq23$. Define $d_1=(n+1)/k$, $d_2=n+1$, and $d_3=1$.
Then the curve $\cC_1$ given in Theorem {\rm \ref{equations}} is not Galois covered by the Hermitian curve $\cH$ over $\FF_{n^6}$.
\end{theorem}

\begin{proof}
By the choise of $d_1$, $d_2$, and $d_3$, the genus of the curve $\cC_1$ can be computed as in Case \ref{Hcase2}.

By separating the cases $3\mid(n+1)$ and $3\nmid(n+1)$ and arguing as in the proof of Theorem \ref{NotGaloisCovered}, it is proved that a possible Galois covering has degree $kn$.

Suppose such a covering exists and $\cC_1\cong\cH/N$ with $N\leq\aut(\cH)$, then the same argument as in Theorem \ref{NotGaloisCovered} allows to apply \cite[Th. 4.4]{GSX} and yields the following expression for $k$:
$$ k=\frac{\left(1+\gcd(3,k)\right)p^{3u}+p^{3u-w}-p^{3u-v}-p^{2u}-\gcd(3,k)p^{u}}{p^{3u-w}-p^{u}-2}, $$
where $n=p^u$ with $p$ prime, and $v,w$ are non-negative integers satisfying $u=v+w$.
But
%MG
 a case analysis shows that this fraction cannot be an integer.
\end{proof}

\begin{theorem}\label{NotGaloisCoveredIII}
Let $n$ be a prime power, $\g$ a divisor of $n+1$, $\gd$ a divisor of $n^2-n+1$, and define $c=(n+1)/\g$, $d=(n^2-n+1)/\gd$. Suppose that one of the following holds:
\begin{itemize}
\item $n=5$, $\g=2$, and $\gd=1$;
\item $n\geq7$, $\g\leq2$, and $\gd\leq(\sqrt{2\g n+1}-1)/2$;
\item $n\geq7$, $\g>2$, and $\g\gd(\g\gd-\gd-1)<n$.
\end{itemize}
Then the curve $\cX/K$ with equations \eqref{XsuK} is not Galois covered by the Hermitian curve $\cH$ over $\FF_{n^6}$.
%Then a possible covering $\varphi:\cH\rightarrow\cX'/K'$, where $\cH$ is the Hermitian curve over $\FF_{n^6}$ and $\cX'/K'$ is the curve with equations \eqref{XsuK}, has degree $\deg(\varphi)=\g\gd n$.
\end{theorem}

\begin{proof}
By arguing as in the proof of Theorem \ref{NotGaloisCovered}, it is proved that a possible Galois covering has degree $\g\gd n$.

Suppose such a covering exists and $\cX/K\cong\cH/N$ with $N\leq\aut(\cH)$, then the same argument as in Theorem \ref{NotGaloisCovered} allows to apply \cite[Th. 4.4]{GSX} and 
%MG
yields the following identity:
\begin{equation}\label{condition}
\gd\left[p^{3u}-\g p^{3u-w}+\left(\gcd(2,\g)-1\right)p^{u}-\g+\gcd(2,\g)\right]=-p^{3u}+p^{3u-v}-p^{3u-w}+p^{2u},
\end{equation}
where $n=p^u$ with $p$ prime, and $v,w$ are non-negative integers with $u=v+w$. By case analysis, it can be shown that \eqref{condition} contradicts the hypothesis on the integers $\g$ and $\gd$.
\end{proof}

\end{document}